\documentclass[12pt]{amsart}
\usepackage{amssymb}
\usepackage{enumerate}

\usepackage[cp1251]{inputenc} 
\usepackage[a4paper, total={6in, 9in}]{geometry}

\DeclareMathOperator{\mmod}{mod}
\DeclareMathOperator{\ord}{ord}
\DeclareMathOperator{\chars}{char}
\DeclareMathOperator{\Res}{Re}

\DeclareMathOperator{\rank}{rank}
\DeclareMathOperator{\vol}{vol}

\newtheorem{lemma}{Lemma}
\newtheorem{theorem}{Theorem}
\newtheorem*{theorem*}{Theorem}

\newcommand*{\bfrac}[2]{\genfrac{}{}{0pt}{}{#1}{#2}}
\newcommand*{\h}{\hat{h}}{}{}{}

\global\long\def\le{\leqslant}
\global\long\def\ge{\geqslant}
\global\long\def\veps{\varepsilon}
\global\long\def\k{\kappa}
\global\long\def\l{\lambda}
\global\long\def\R{\mathbb{R}}
\global\long\def\F{\mathbb{F}}
\global\long\def\C{\mathbb{C}}
\global\long\def\Z{\mathbb{Z}}
\global\long\def\P{\mathbb{P}}
\global\long\def\G{\mathbb{G}}
\global\long\def\Q{\mathbb{Q}}

\global\long\def\Cc{\mathcal{C}}
\global\long\def\Oo{\mathcal{O}}
\global\long\def\M{\mathcal{M}}

\newtheorem{remark}{Remark}
\newtheorem{corollary} {Corollary}

\begin{document}

\title[A. Sedunova]{Bounds for the integral points on elliptic curves over function fields}

\author[A. Sedunova]{Alisa Sedunova}
\address{
Max Planck Institute, Bonn, Germany}
\email{alisa.sedunova@phystech.edu}

\date{}

\maketitle

\begin{abstract}
In this paper we give an upper bound for the number of integral points on an elliptic curve $E$ over $\F_q[T]$ in terms of its conductor $N$ and $q$. We proceed by applying the lower bounds for the canonical height that are analogous to those given by Silverman and extend the technique developed by Helfgott-Venkatesh to express the number of integral points on $E$ in terms of its algebraic $\rank$. We also use the sphere packing results to optimize the size of an implied constant. In the end we use partial Birch Swinnerton-Dyer conjecture that is known to be true over function fields to bound the algebraic $\rank$ by the analytic one and  apply the explicit formula for the analytic rank of $E$.
\end{abstract}

\section*{Introduction}
Let $q$ be a prime power and $K=\F_q[T]$ be the field of polynomials in formal variable $T$ with coefficients in a finite field $k=\F_q$ of order $p$.
Our main goal here is to prove the following theorem.
\begin{theorem} \label{maintheorem}
Let $E$ be an elliptic curve over $\F_q[T]$ of a conductor $N$. 
Assume that the integral points on $E$ are on minimal model. Then the number of integral points on $E$ satisfies
$$\#E(\F_q[T]) \leq  \exp \left( c \frac{\deg N_E}{ \log \deg N_E}\right),$$
where $c$ is an absolute constant and $N_E$ is the degree of the conductor of $E$.
\end{theorem}

Notice, that we work in the context where the analogue of Siegel's theorem is true (it is proven in \cite{voloch}). In particular, if $E$ is an ellipic curve over $\F_q[T]$ parametrized by $a,b \in \F_q$, then $E(\F_q(T)) = E(\F_q)$ and $\# E(\F_q[T]) \leq q+1+2\sqrt{q}$.
For a more general function field $\F_q(C)$ with ring of integers $A$ we can have $E(A)$ infinite.
Notice that if $E$ is constant, i.e. defined over $\F_q$, then $E(\F_q(T)) = E(\F_q)$, therefore Siegel theorem holds in this case too. For the case of $E$ being isotrivial (not defined over $F_q$ and supersingular) Siegel theorem may be false.

The tools that allow us to proceed are that the necessary part of the famous Birch and Swinnerton-Dyer conjecture holds in the function field context, as well as the bounds for the analytic rank over a function field are known, thanks to the explicit formula given by Brumer in \cite{Brumer1992}. We also extend the technique of Helfgott (see \cite{Helfgott2004}) to obtain an upper bound for the number of integral points on $E$ in terms of its algebraic rank. However, this brings us to results that do depend on the curve. To get rid of this dependence we have to work with the estimation of the sort $\#E(\F_q[T]) \ll c^{\rank E+m}$ more carefully(here $m$ stands for the number of multiplicative places). Namely, we extend the method developed by Helfgott-Venkatesh in \cite{Helfgott2006} based on the ideas of Silverman \cite{Silverman}. We optimize the size of $c$ by applying sphere packing results of Kabatiansky and Levenshtein \cite{Kabatjanskii1978}.

The previously known bounds of such a type (see Theorem 1 of \cite{bombieripila}) give us $\#E(\Z \cap I^2) \ll |I|^{\frac{1}{3}+\veps}$, where we are restricted to counting integral points lying in a small box of size $|I|$, where $I$ is an 'interval' of polynomials defined in \cite{bombieripila}. This result is analogous to Bombieri-Pila theorem \cite{Bombieri1989}, that gives the upper bound $\ll N^{\frac{1}{d} + \veps}$, where $d$ is the degree of a curve and is equal to $3$ in the case of elliptic curves, however the method of getting it is different and mainly based on the ideas of Helfgott-Venkatesh \cite{Helfgott2006} and the interpolation part used by Heath-Brown \cite{Heath-Brown1994}. Here we take the approach proposed by Helfgott in \cite{Helfgott2006} and further developed by Helfgott-Venkatesh in \cite{Helfgott2006}, but it turns out that this way of doing things is closely related to the one used in \cite{Bombieri1989}.

The paper is organized as follows.
In Section 2 we review some basic definitions, that are going to be used throughout the paper as well as some important facts (see (\ref{bsdff}) and (\ref{brumer}), also (\ref{anbound})) that are crucial in our proof. Then we prove several standard results regarding canonical height on an elliptic curve $E$. Based on this we show how to get a cheap, but useless bound for the number of points in $E(\F_q[T])$ of a bounded height. We introduce local heights $\l_v(\cdot)$ to get rid of this problem and prove lower bounds for $\l_v(\cdot)$ under some 'good' slicing, that will bring us to another bound for the canonical height, namely Lemma \ref{proposition34}, that is proved in the spirit of \cite[Proposition 3.4]{Helfgott2006}. We also need a lower bound for the canonical height on $E$ due to Silverman, see \cite{Silverman}.

Further, in Section 3 we prove the bound for the number of $S$-integral points on $E$ in terms of algebraic rank of $E$ using Lemmas from previous sections together with sphere packing results by Kabatiansky and Levenstein \cite{Kabatjanskii1978}.
Finally, in Section 5 we prove the main result by taking an advantage of working in function fields, where Birch and Swinnerton-Dyer conjecture partly holds (see (\ref{bsdff})) and apply the explicit formula for an analytic rank, given in the expression (\ref{brumer}) by Brumer.

\section*{Auxiliary results}
We briefly review some tools that we use during the proof. For more detailed survey see the work of Ulmer \cite{Ulmer2011}.
Let $k=\F_q$ be the finite field of cardinality $q$, with its characteristics $\chars(k)=p$. We write $K$ for the function field of a smooth, projective absolutely irreducible curve $\Cc$ over $k$. In what follows we consider $\Cc=\P^1$, thus $K=\F_q[T]$ is the field of polynomials in a formal variable $T$ with coefficients lying in $k$. For $X\in K$ we denote by $|X|$ its norm: $|X|=q^{\deg X}$. 
We recall that an elliptic curve over $K$ is a smooth, projective, absolutely irreducible algebraic curve of genus 1 over $K$ with a $K$-rational point $\Oo$ that plays the role of identity element in the group $E(K)$ of $K$-rational points lying on $E$ (Mordell-Weil group of $E$). Lang and N\'{e}ron generalized the result of Mordell-Weil and proved that for a funstion field $K$ $E(K)$ is a finitely generated abelian group. As a consequence of this result the torsion group $E(K)_{\text{tors}}$ (i.e. the group of $K$-points on $E$ of finite order) is finite and isomorphic to a group of the form 
$$\Z/m\Z \times \Z/n\Z,$$
where $m$ divides $n$ and $p$ does not divide $m$. 
Define an algebraic $\rank(E)$ of an elliptic curve $E/K$ as the number of independent points of infinite order in $E(K)$, so to say  the number of copies of $\Z$ in $E(K)$. 

An equivalent definition of an elliptic curve $E/K$ can be given due to the Riemann-Roch theorem: an elliptic curve $E/K$ can always be described as a projective plane curve of degree 3 with a (homogeneous) Weierstrass equation
\begin{equation}\label{homWei}
y^2z+a_1 xyz+ a_3yz^2=x^3+a_2x^2z+a_4xz^2+a_6z^3,
\end{equation}
where all $a_i$ belong to $K$. As usually, the origin is the point at infinity, namely $\Oo=[0:1:0]$. The condition of smoothness of $E$ is equivalent to the fact that its discriminant $\Delta$ is not zero. 
The equation above can be also given in an affine form by the change of variables $(x,y) \to (x/z,y/z)$
\begin{equation}\label{Wei}
y^2+a_1 xy+ a_3y=x^3+a_2x^2+a_4x+a_6.
\end{equation}
Let $v$ be an equivalence class of valuations of $K$. 
Recall that a valuation on a field $K$ is a generalization of the $p$-adic norm. Concretely, it is a function $|\cdot|_v$ from a field $K$ to the real numbers $\R$ such that the following properties hold for all $x,y \in K$: 
\begin{itemize}
\item[$\circ$]{$|x|_v \geq 0$, $|x|=0$ if and only if $x=0$;}
\item[$\circ$]{$|xy|_v=|x|_v \cdot |y|_v$;}
\item[$\circ$]{$|x|_v \leq 1$ implies $|1+x|_v \leq C$ for some constant $C \geq 1$ independent of $x$.}
\end{itemize}
Notice that if a valuation $|\cdot|_v$ satisfies the last condition above with $C=2$, then it satisfies the triangle inequality $$|x+y|_v \leq |x|_v+|y|_v$$ for all $x,y \in K$ and such a valuation is called archimedean. If the condition is satisfied with $C=1$, then $|\cdot|_v$ satisfies the stronger ultrametric inequality: $$|x+y|_v \leq \max(|x|_v,|y|_v)$$ for all $x,y \in K$ and we call this valuation non-archimedean. Here we work only with non-archimedean valuations.

For every $v$ denote by $\Oo_{(v)}$ the ring of rational functions on $\Cc$ regular at $v$. In our case ($\Cc=\P^1$) the finite places correspond to monic irreducible polynomials $f\in K=\F_q[T]$. If such a place $v$ corresponds to $f$, then $$\Oo_{(v)} = \{g/h. \text{ s.t. } g,h\in K, \deg(g) < \deg(h)\}.$$ Assume that the degree of $v=\infty$ is 1. Write $\M_v \subset \Oo_{(v)}$ for the maximal ideal (its elements are the functions vanishing at $v$) and $\k_v=\Oo_{(v)}/\M_v$ for the residue field at $v$. Set $\deg(v)=[\k_v:k]$, $q_v=q^{\deg(v)}$ for the norm of $v$.
Choose a minimal integral model for $E$ in the form (\ref{Wei}). Let $\bar{a}_i \in \k_v$ be the reductions of the coefficients at $v$ and define the reduced curve $E_v$ by 
\begin{equation}\label{Weired}
E_v: \;y^2+\bar{a}_1 xy+ \bar{a}_3y=x^3+\bar{a}_2x^2+\bar{a}_4x+\bar{a}_6
\end{equation}
over the residue field $\kappa_v$. 
We say that $E_v$ has
\begin{itemize}
\item[$\circ$]{a good reduction at $v$ if $E_v$ defines an elliptic curve over $\k_v$ ($v \nmid \Delta$),}
\item[$\circ$]{a multiplicative (nodal) reduction at $v$ if $E_v$ has a node at $v$. If the tangent lines at the node are rational over the residue field $\kappa_v$, then we call this type of reduction split multiplicative. Otherwise non-split multiplicative.}
\item[$\circ$]{an additive (cuspidal) reduction at $v$ if $E_v$ has a cusp at $v$.}
\end{itemize}
Notice that terms multiplicative and additive are used here to emphasize that the non-singular part of the reduced curve defined by $E_v^*=E_v/\{\text{singular point}\}$ is isomorphic to $\G_m$ (or $\G_m[\cdot]$ for the non-split case) and $\G_a$ respectively (here $\G_m$ stands for the multiplicative group, $\G_m[\cdot]$ for the twisted multiplicative group and $\G_a$ for the additive group). Elliptic curves, $\G_a$, $\G_m$ and $\G_m[\cdot]$ over $K$ are the only irreducible algebraic curves over $K$ having group structures given by regular maps. 

The reduced curve $E_v$ may be singular, but yet the set of nonsingular points of $\tilde{E_v}(K_v)$ forms a group. Moreover $E(K)$ admits the following filtration of abelian groups
$$E_1(K)\subset E_0(K)\subset E(K),$$
where 
$E_0(K)=\{P \in E(K): P_v \in \tilde{E_v}(K_v)\}$ and $E_1(K)=\{P \in E(K): P_v=O_v\}$ with $P_v$ taken to be the image of $P \in E(K)$ under the reduction map $E(K) \to \tilde{E_v}(K_v)$.

A model for $E$ given by $E_v$ with its coefficients $\bar{a}_i \in \Oo_{(v)}$ is called integral at $v$. The minimal integral model at $v$ is the model $E_v$ with the valuation of the discriminant $\Delta$ of $E$ being minimal.
The local exponent $n_v$ of the conductor at $v$ is given by
\begin{equation*}
n_v=
\begin{cases} 0, & \mbox{if }E \mbox{ has good reduction at }v, 
\\ 1, & \mbox{if }E \mbox{ has multiplicative reduction at }v,
\\ 2+\delta_v, & \mbox{if }E \mbox{ has additive reduction at }v,
\end{cases}
\end{equation*}
where $\delta_v$ is the wild ramification
\begin{equation*}
\delta_v=
\begin{cases} 0, & \mbox{if }p>3, 
\\ \ge 0, & \mbox{if }p=2,3.
\end{cases}
\end{equation*}
Thus $n_v$ has the information about the ramification in the field extensions generated by the points of finite order in the group law of the elliptic curve $E$. The conductor of $E/K$ is given by a product of prime ideals and associated exponents $n_v$. The (global) conductor of $E$ is a divisor $N=\sum_v n_v [v]$. The degree of the conductor is $\deg N = \sum_v n_v \deg v$.
$N$ is an effective divisor on $\P^1$ which is divisible only by the places $v$ of bad reduction of $E$.
The $L$-function of $E$ is defined be the Euler product
\begin{equation} \label{lfun}
L(E,s)=\prod_{v\nmid N}^{\text{good}}\left(1-\frac{a_v}{q_v^{s}}+\frac{q_v}{{q_v}^{2s}}\right)^{-1} \times
\prod_{v\mid N}^{\text{mult}}\left(1-\frac{1}{q_v^{s}}\right)^{-1} 
\end{equation}
where "good" stands for "$E$ has a good reduction at $v$", "mult" -- for the case of either split multiplicative or non split multiplicative reduction at $v$
and, finally, $a_v$ is an integer defined as
\begin{equation*}
a_v=
\begin{cases} q_v+1-\#E_v(k_v), &\mbox{if }E\mbox{ has good reduction at }v, 
\\ \pm 1, &\mbox{if }E\mbox{ has multiplicative reduction at }v,
\\ 0, &\mbox{if }E\mbox{ has additive reduction at }v.
\end{cases}
\end{equation*}
($a_v=1$ for the split multiplicative reduction and to $-1$ for the non split multiplicative reduction).
Due to the Hasse bound on $a_v$ the first product of (\ref{lfun}) converges absolutely for $\Res s > 3/2$ and admits a meromorphic continuation on $\C$.
As usually we define an analytic rank of $E/K$ as the order of vanishing of its $L$-function at $s=1$
\begin{equation*}
\rank_{an} (E) = \ord_{s=1} L(E,s).
\end{equation*}
We recall that an elliptic curve $E/K$ is called constant if it can be defined by a Weierstrass equation (\ref{Wei}) with coefficients belong to $k$. It is called non-constant if it is not constant. Also $E/K$ is called isotrivial if it becomes constant over some finite extension of $K$, otherwise -- non-isotrivial.
\begin{remark} In the non-constant case of $E$ Theorem 9.3 of \cite{Ulmer2011} gives us an upper bound of a type $\rank_{an} E \le N$.
\end{remark}
The famous conjecture of Birch and Swinnerton-Dyer connects the analytic behaviour of $L$-functions of elliptic curves with the group of $K$-rational points on $E/K$, in particular (among some other relations) it predicts that 
$$\rank_{an} (E) \stackrel{?}{=} \rank(E).$$
While the original conjecture remains unsolved, much more is known in this context for the case of function fields.
\begin{theorem*} [Tate \cite{Tate1966}, Milne \cite{Milne}]
Let $E$ be an elliptic curve over a function field $K$. Then 
\begin{equation} \label{bsdff}
\rank E \leq \rank_{an} E.
\end{equation}
\end{theorem*}
The usual technique for obtaining upper bounds of an analytic rank is using so-called explicit formula. We refer here to the result given by \cite{Brumer1992}.
\begin{theorem*} [Brumer \cite{Brumer1992}]
Let $E$ be an elliptic curve over $\F_q[T]$. Then its analytic rank is bounded by
\begin{equation} \label{brumer}
\rank_{an} E \leq \frac{(b_E -4)\log q}{2 \log b_E} + O \left(\frac{n_E \log^2 q}{\sqrt{q} \log^2 b_E}\right),
\end{equation}
where $b_E$ is the degree of $L$-function as a polynomial in $q^{-s}$.
\end{theorem*}
For the case of $\F_q[T]$ we have 
$$b_E=n_E-4,$$
where $n_E=\deg N$ and $N$ is the conductor of an elliptic curve $E/K$.
We note that if $E$ has $a$ additive reductions and $m$ multiplicative reductions, then
$$n_E \leq 2a+m.$$
This result is interesting if and only if $n_E$ is rather big, since the trivial bound for the rank is $n_E + 4g_X-4$. We thus have
\begin{equation} \label{anbound}
\rank_{an} E \leq \frac{(\deg N - 8)\log q}{2 \log \deg N} + O \left(\frac{\deg N \log^2 q}{\sqrt{q} \log^2 \deg N}\right).
\end{equation}
The easy bound is
$$\rank E \leq \rank_{an} E \leq b_E = n_E-4.$$
If $E$ is constant, then $\rank E = 0$.

\section*{Heights and its properties}

Here we investigate some properties of height function on an elliptic curve $E$ over a field $K=\F_q[T]$. The crucial fact here is that $|\h-\frac{1}{2}h_x|$ and $|\h^E-\frac{1}{3}h_y|$ are bounded on the set of all points of $E$. This allows us to give a lower bound for $\h^E(P)$ as well as to estimate the number of points with $\h^E < c_2$ under condition that $E$ does not have any non torsion points $P$ with $\h^E(P) > c_1$. However, this path leads us to a problem that the bound would depend on the curve. To avoid this difficulty we will use local heights as in \cite{GrossSilv} and establish the bound $\l_v(P-Q) \geq \min(\l_v(P),\l_v(Q))$ that fails only in the case of bad reduction with which we will deal separately. We subdivide $E(K_v)$ into small enough number of slices, so that $\l_v(P-Q) \geq \min(\l_v(P),\l_v(Q))$ still holds true on these slices with $P$, $Q$ belong to the same slice (for more details see Lemma \ref{lambda>min} and Lemma \ref{slicinglambda}). Using that we prove that integral points we wish to count are far apart from each other in the Mordell-Weil lattice.
Recall that any elliptic curve over $K$ can be written in the following form 
\begin{equation} \label{weif}
E:\; y^2 = f(x),
\end{equation}
where $f(x) \in K$ is a cubic polynomial defined by Weierstrass equation.
We say that $d \in K$ is square free if it has no factor of the form $g^2$ with $g \in K$ and $\deg g \geq 1$.
For any $d \in K$ square free define a quadratic twist of $E$ as
\begin{equation} \label{twist}
E_d:\; dy^2 = f(x).
\end{equation}
Note that we restrict to the case of square free $d$, since if $d$ has a squared factor, then by a change of variables in
(\ref{twist}) one can find a curve $E_d^*$ isomorphic to $E_d$.
We write $\h^E$ for the canonical height on an elliptic curve $E$, and $h_x$, $h_y$ for the height on $E$ with respect to $x$ and $y$:
\begin{equation} \label{canh}
\h^E((x,y))= \lim_{n \to \infty} \frac{1}{n^2} h_x([n](x,y)),
\end{equation}
where we use the notation $[n]P = \underbrace{P+\ldots+P}_{n \text{ times}}$
and
\begin{equation*}
\begin{split}
&h_x((x,y))=
\begin{cases} 0, & \mbox{if }P=\Oo, 
\\ \log_q H(x), & \mbox{otherwise},
\end{cases}\\
&h_y((x,y))=
\begin{cases} 0, & \mbox{if }P=\Oo, 
\\ \log_q H(y), & \mbox{otherwise}.
\end{cases}
\end{split}
\end{equation*}
For any $x \in K$ define its norm by $|x|=q^{\deg x}$. 
We notice that $\h^E$ is defined on all points of $E(\bar{K})$ and $\h$ is a positive definite quadratic form on $E(\bar{K})$ as well as on $E(K)$(in the sense that it maps non-torsion elements to positive numbers).

For $x=x_0/x_1$ with $x_0,x_1 \in K$ not having as polynomials any common factor other than a constant polynomial in $K$ (we encrypt this fact by $(x_0,x_1)_K={1}$), one can write $H(x)=\max(|x_0|,|x_1|)$. Let $L$ be any algebraic field extension of $\F_q[T]$. Define $H(y)$ by
\begin{equation*}
H(y)=(H_L(y))^{[L:K]^{-1}},\; H_L(y)=\prod_w \max (|y|_w^{n_w},1),
\end{equation*}
where $y \in L$, the product is taken over all places $w$ of $L$, $n_w$ stands for the degree of quotient field $L_w/K_w[T]$.
For example, if $y=\frac{y_0}{y_1}$ with $y_0,y_1 \in K$, then $y \in \F_q(T)$ and for $L=\F_q(T)$ $H(y)=H_L(y)=\max(|y_0|,|y_1|)$.
We list some important properties of the canonical height in the following lemma.
\begin{lemma} \label{hprop}
Let $f(x) \in K=\F_q[T]$ be a monic polynomial of non-zero discriminant in (\ref{weif}). Let also $d$ be a square-free polynomial $d \in K$ and $P=(x,y)$ be a $K$-point on the quadratic twist $E_d$ of $E$. Let $P'=(x,d^{1/2}y)$ be a point on $E_1=E$ associated to $P$.
Then
\begin{enumerate}
\item{$\h^{E_d}(P)=\h^E(P')$, where the canonical heights are defined on $E_d$ and $E$, respectively and, of course, $\deg f=3$.}
\item{The height $h_y$ $(y \neq 0)$ is bounded on $E$, namely
$h_y(P') \geq \frac{3}{8}\deg d.$}
\item{If $\deg f =3$, then $\h^{E_d}(P) \geq \frac{1}{8}\deg d + c_f$, where $c_f$ is a constant depending only on $f$.}
\end{enumerate}
\end{lemma}
\begin{proof}
{\bf 1.} We do not put any change in the $x$-coordinate, so clearly $h_x(P)=h_x(P')$. For the sake of simplicity we consider the case of $\chars k \neq 2,3$. The proof goes analogously in the characteristics 2 and 3. Under this assumption we can write an equation of $E$ in so-called short Weierstrass form (see, for example, Theorem 2.1 in \cite{Milne2006})
\begin{equation}\label{weishort}
E: y^2=x^3+ax+b,\; a,b \in K.
\end{equation}
Then the duplication law on $E$ is given by 
\begin{equation} \label{duplic}
[2]P=P+P=\left(\frac{(3x^2+a)^2-8xy^2}{4y^2}, \frac{F_{a,b}(x)}{(2y)^3}\right),
\end{equation}
where $F_{a,b}(x)=x^6+5ax^4+20bx^3-5a^2x^2-4abx-a^3-8b^2$.
The short Weierstrass equation for the twisted curve $E_d$ is given by the change of variables $(x,y) \to (dx,d^2y)$
$$E_d: y^2=x^3 +ad^2 x +bd^3.$$
Write $X(P)$ and $Y(P)$ for the coordinate functions of $P$. Then
\begin{equation*}
X([2]P')=\frac{(3x^2+a)^2-8dxy^2}{4dy^2} \text{ and } X([2]P)=\frac{(3x^2+a)^2-8dxy^2}{4y^2}.
\end{equation*}
Thus $X([2]P')=X((P+P)')$.
Further,
\begin{equation*}
Y([2]P)=\frac{F_{a,b}(x)}{(2y)^3} \text{ and } Y([2]P')=\frac{F_{a,b}(x)}{d^{\frac{3}{2}}(2y)^3},
\end{equation*}
which shows that $Y([2]P')=Y((P+P)')$. We conclude that $(P+P)'=P'+P'$. Notice that here the addition is made on $E_d$ on the left hand side and on $E$ on the right hand side.
Iterating this and using (\ref{canh}) we get
$$\h(P)=\frac{1}{2}\lim_{n \to \infty} \frac{h_x([2^n]P)}{2^{2n}} = 
\frac{1}{2}\lim_{n \to \infty} \frac{h_x([2^n]P')}{2^{2n}} = \h(P').$$
{\bf 2.}
Write $y=\frac{y_0}{y_1}$ for $y_0,y_1 \in K$, such that they do not have any common factor $g \in K$ of a positive degree. For $a,b \in K$ we denote by $\langle a,b \rangle =\langle a,b \rangle_K$ the biggest common factor (in the sense that there is no other polynomial $g \in K$ of a bigger degree, such that $g$ is a factor of both $a$ and $b$) of polynomials $a,b$. We have $\langle y_0,y_1 \rangle_{K}=1$ and we call such polynomials coprime. 
If $g$ is a monic irreducible polynomial, such that $g$ is a factor of $\langle d,y_1^2 \rangle$, then $g^2$ can not be a factor of $\langle d,y_1^2 \rangle$ (by the fact that $d$ is a square free polynomial), but it is a factor of $y_1$. 
Hence, if $g$ is not a factor of $\langle d, y_1^2 \rangle$, then write
$$\langle d, y_1^2 \rangle = \frac{d y_1^2}{\{d,y_1^2\}},$$
where $\{d,y_1^2\}$ is a minimal polynomial that has both $d$ and $y_1^2$ as factors. Then using the fact that $y_0$ and $y_1$ are taken to be coprime we conclude that $g$ has a power $-1$ as a factor of $dy^2= dy_0 y_1^{-2} = d^2 y_0^2 \langle d,y_1^2\rangle^{-1}\{d,y_1^2\}^{-1}$. Recall that $P$ lies on our curve $E$, so $dy^2=f(x)$ and if $g$ has a non-negative degree as a factor of $x$, then it also has a non-negative degree as a factor of $dy^2$. But if $g$ has a negative degree as a factor of $x$, then its degree in $dy^2$ drops to $\leq -3$ leaving us with a contradiction.
Therefore we conclude that $|y_1| \geq \langle d, y_1^2\rangle^2$.
Since $y \in K$ we can write by the definition of $H(y)$ and considering the Euclidean norm
\begin{equation*}
\begin{split}
H(y)&=\max\left(|y_0||d^{-1}\langle d,y_1^2 \rangle|^{-\frac{1}{2}},|y_1||\langle d,y_1^2 \rangle|^{-\frac{1}{2}}\right) \\
& \geq \max \left(|y_0||d^{-1}\langle d,y_1^2 \rangle|^{-\frac{1}{2}},|\langle d,y_1^2 \rangle|^{\frac{3}{2}}\right)
 \geq |d|^{\frac{3}{8}},
\end{split}
\end{equation*}
where we used the fact that $\max$ gets its minimal value when $|\langle d,y_1^2 \rangle|=|d|^{\frac{1}{4}}$.
Finally, $h_y(P)=\log H(P) \geq \frac{3}{8} \log q^{\deg d} = \frac{3}{8} \deg d$.

{\bf 3.} It is a simple consequence of 1 and 2. If $P'$ is a point on $E=E_1$, then, by 1 $\h^{E_d}(P)=\h^{E}(P')$. The difference $|\h^E-h_x^E|$ is bounded on $E$, thus by application of second part of \ref{hprop} the result follows.
\end{proof}
\begin{corollary}
Let $E$ be an elliptic curve over $K=\F_q[T]$. If there are no non-torsion points $P \in E(K)$ of a canonical height $\h(P) > c_1$, then there are at most 
$$O\left(\left(1+2\sqrt{\frac{c_2}{c_1}}\right)^{\rank E}\right)$$
points in $E(K)$ of a canonical height $<c_2$.
\end{corollary}
\begin{proof}
Let's take our canonical height to the square of the Euclidean norm. There is one to one correspondence $f: K^{\rank E} \to K^{\rank E}$ such that $\h^E(\P)=|f(\P)|^2$ for all vectors $\P \in K^{\rank E}$ of the length $\rank E$ with coordinates in $K$. Since $\h^E(P) > c_1$ for all non-zero $P \in K$, then we are equipped by $f(K^{\rank E})$ with a lattice $L$, such that for every element $l \in L$ different from $0$ we have $|l| \geq c_1^{\frac{1}{2}}$. For every point $l \in L$ draw a sphere $Sp_l$ centred at $l$ of the radius $\frac{1}{2}c_1^{\frac{1}{2}}$, so that they do not overlap. Each of the spheres $Sp_l$ is contained in the bigger one $Sp$ with the radius $c_2^{\frac{1}{2}}+\frac{1}{2}c_1^{\frac{1}{2}}$ centred at the origin.
By bounding the total volume of all spheres by $\vol(Sp) \leq (c_2^{\frac{1}{2}}+\frac{1}{2}c_1^{\frac{1}{2}})^{\rank E}$ we end the proof.
\end{proof}
The implied constants $c_1$, $c_2$ do not have any dependency on the twist, but depend on the curve. This would bring us to a problem once we want to bound the canonical height in terms of naive height (namely, we want something of the sort $h(P) \leq c_3$, where $h(P)$ is the naive height and $c_3$ is an absolute constant), because then the constant inside big $O$ will change to $(1+2 \sqrt{c_3 / c_1})^{\rank E}$, where $c_1$ depends only on the curve, whilst $c_3$ depends on both the curve and $c_2$ (say, $c_2=c_3+O_E(1)$). To avoid this difficulty we have to exclude the hidden dependency by the method proposed in \cite{Helfgott2006}.

Recall that $\k_v$ is the residue field at $v$ and $d_v=\deg(v)=[\k_v:k]$. Let $M_k$ be the set of places $v$ on $K$. For each place $v \in K$, there exists a natural local height function $\l_v$ such that the canonical height on $E$ can be given in terms of $\l_v$
$$\h^E(P) = \frac{1}{[K:\Q]} \sum_{v \in M_K}d_v \l_v(P).$$
We say that an elliptic curve $E$ over a non-archimedean local field $K$ has potentially good reduction if it has a model with good reduction in some extension of $K$. Similarly, $E$ has potentially multiplicative reduction if it does not have potentially good reduction.

\begin{lemma}\label{lambda>min}
Let $E$ be an elliptic curve over a non-archimedean local field $K_v$ with potentially good reduction. Let $P,Q \in E(K_v)$ be two distinct points. Then
$$\l_v(P-Q) \geq \min(\l_v(P),\l_v(Q)).$$
\end{lemma}

\begin{proof}
Consider an extension $L_w$ of $K_v$ on which $E$ has good reduction. Choose a Weierstrass equation for $E$ over $L_w$ such that $v(\Delta)=0$. Then by \cite[Proposition 2]{GrossSilv} we find that 
$$\l_v(P)=\l_w(P)=\frac{1}{2} \max (\log|x(P)|_w,0).$$
Since $v$ is non-archimedean, then $|x+y|_v \leq \max({|x|_v,|y|_v})$ and the claim follows.
\end{proof}
The following lemma is \cite[Lemma 3.2]{Helfgott2006} and the proof is completely analogous. 
\begin{lemma} \label{slicinglambda}
Let $E$ be an elliptic curve over a non-archimedean local field $K_v$ with potentially multiplicative reduction. Then for any $\veps > 0$ small enough, there is a subdivision
$$E(K_v)=W_{v,0} \cup W_{v,1} \cup \ldots \cup W_{v,d_v} \ll |\log \veps|,$$
such that for any two distinct points $P,Q \in W_{v,0}$ we have
$$\l_v(P-Q) \geq \min(\l_v(P),\l_v(Q)),\;\;\;\; \l_v(P_1), \l_v(P_2) \geq 0,$$
and  for any two distinct points $P,Q \in W_{v,j}$, where $1 \leq j \leq d_v$ we have
\begin{equation*}
\begin{split}
&\l_v(P-Q) \geq (1-\veps)\max(\l_v(P),\l_v(Q)),\\
&\l_v(P-Q) \geq (1-2\veps)\max(\l_v(P),\l_v(Q)),
\end{split}
\end{equation*}
where the implied constant is absolute.
\end{lemma}

Now we have to adapt \cite[Proposition 3.4]{Helfgott2006}, that will serve us for as a bound for the canonical height that does not depend on the curve any longer. Here we assume that our two points are of the same reduction as well as that they fall into the same $W$-class, so we can apply Lemma \ref{slicinglambda}.

Since we are working in $K=\F_q[T]$, we don't have any archimedean valuations and thus, the proof can be significantly simplified.
\begin{lemma} \label{proposition34}
Let $E$ be an elliptic curve over $K$. Let $S$ be a finite set of places of $K=\F_q[T]$, that includes all irreducible divisors of the discriminant $\Delta$ of $E$. Let $P_1$, $P_2$ be two distinct integral points on $E$ that belong to the same set $W_{v,i}$ for any place $v$ among the ones with potentially multiplicative reduction. Suppose that
$$\sum_{v \in T} d_v |\l_v(P_1)-\l_v(P_2)| \leq \veps \max_{j=1,2} \sum_{v \in T} d_v \l_v(P_j),$$
where $\veps >0$ sufficiently small and 
$$T = \{v \in S: \l_v(P_1), \l_v(P_2) \geq 0\}.$$ 
Assume that $P_1$ and $P_2$ have the same reduction modulo $I$, where $I$ is any ideal not divisible by irreducible elements of $S$. Then
$$\hat{h}(P_1-P_2) \geq (1-2\veps) \max (\hat{h}(P_1),\hat{h}(P_2))+ \frac{\log N I}{[K:L]}.$$
\end{lemma}
\begin{proof}
If $v$ is a finite place of good reduction, then $\l_v(P) \geq 0$. Recall that $S$ contains all places that divide the discriminant $\Delta$ of $E$. Then by definition of a canonical height through local heights we have
\begin{equation*}
\begin{split}
\h(P_1-P_2) &\geq  \sum_{v \in S} d_v \l_v(P_1-P_2) + \sum_{v \notin S} d_v \l_v(P_1-P_2) \\
&=  \sum_{v \in S} d_v \l_v(P_1-P_2) +  \sum_{\bfrac{v \text{ finite}}{v(I) > 0}} d_v \l_v(P_1-P_2).\\
\end{split}
\end{equation*}
We now subdivide our set $S$ as $S = T \cup S/T$, where $T$ is defined in the statement of the lemma. Let us consider two differences
\begin{equation*}
\begin{split}
&\sigma_1=\sum_{v \in T}d_v \l_v(P_1-P_2)-(1-\veps)\sum_{v\in T}d_v \min (\l_v(P_1), \l_v(P_2)),\\
&\sigma_2=\sum_{v \in S/T}d_v \l_v(P_1-P_2)-(1-2\veps)\max_{j=1,2}\sum_{v\in S/T}d_v \l_v(P_j).
\end{split}
\end{equation*}
The goal now is to show that these two quantities $\sigma_1, \sigma_2 \geq 0$. Once we are done it remains to consider only finite places $v$, such that $v(I) > 0$.
We use the following notations $\sum^{\text{good}}, \sum^{0}, \sum^{j}$ denote that $P_1,P_2$ are of potentially good reduction, potentially multiplicative reduction and fall into $W_{v,0}$, potentially multiplicative reduction and fall into $W_{v,j}$ with $j > 0$ respectively. 
By Lemma \ref{lambda>min} and Lemma \ref{slicinglambda} we have
\begin{equation*}
\begin{split}
\sigma_1 &\geq  \sum_{v \in T}^{good,0}d_v \min_{j=1,2}\l_v(P_j)+(1-\veps)\sum_{v \in T}^{j} d_v \max_{j=1,2}\l_v(P_j) -(1-\veps)\sum_{v \in T}d_v \min_{j=1,2}\l_v(P_j)\\
&= \veps \sum_{v \in T}d_v \min_{j=1,2}\l_v(P_j) - \veps \sum_{v \in T}^{j}d_v \max_{j=1,2}\l_v(P_j)
+\sum_{v \in T}^{j}d_v \big(\max_{j=1,2}\l_v(P_j)-\min_{j=1,2}\l_v(P_j)\big)\\
&\geq \veps \sum_{v \in T}d_v \min_{j=1,2}\l_v(P_j) - \veps \sum_{v \in T}^{j}d_v \max_{j=1,2}\l_v(P_j)\\
&=\veps \sum_{v \in T}^{good,0}d_v \min_{j=1,2}\l_v(P_j)+\veps \sum_{v \in T}^{j}d_v \big(\min_{j=1,2}\l_v(P_j)-\max_{j=1,2}\l_v(P_j)\big)\\
&=\veps \sum_{v \in T}^{good,0}d_v \min_{j=1,2}\l_v(P_j)-\veps \sum_{v \in T}^{good,0}d_v\big(\min_{j=1,2}\l_v(P_j)-\max_{j=1,2}\l_v(P_j)\big)\\
&+\veps \sum_{v \in T}d_v \big(\min_{j=1,2}\l_v(P_j)-\max_{j=1,2}\l_v(P_j)\big)\\
&=\veps \sum_{v \in T}^{good,0}d_v\max_{j=1,2}\l_v(P_j)+\veps \sum_{v \in T}d_v \big(\min_{j=1,2}\l_v(P_j)-\max_{j=1,2}\l_v(P_j)\big).\\
\end{split}
\end{equation*}
Now we apply the assumption of our lemma and get
\begin{equation*}
\begin{split}
\sigma_1 &\geq \veps \sum_{v \in T}^{good,0}d_v\max_{j=1,2}\l_v(P_j)-\veps^2 \sum_{v \in T}d_v\max_{j=1,2}\l_v(P_j)\\
&=(\veps-\veps^2) \sum_{v \in T}^{good,0}d_v\max_{j=1,2}\l_v(P_j)-\veps^2  \sum_{v \in T}^{j}d_v\max_{j=1,2}\l_v(P_j) \geq 0
\end{split}
\end{equation*}
by choosing $\veps$ small enough. Applying the same condition again we get
\begin{equation*}
\begin{split}
\sum_{v \in T}d_v \l_v(P_1-P_2) &\geq (1-\veps) \sum_{v \in T} d_v\max_{j=1,2}\l_v(P_j)\\
&+(1-\veps)\sum_{v \in T}d_v\big(\min_{j=1,2}\l_v(P_j)-\max_{j=1,2}\l_v(P_j) \big)\\
&\geq  (1-\veps) \sum_{v \in T} d_v\max_{j=1,2}\l_v(P_j) + \sum_{v \in T}d_v\big(\min_{j=1,2}\l_v(P_j)-\max_{j=1,2}\l_v(P_j) \big) \\
&\geq (1-\veps) \sum_{v \in T} d_v\max_{j=1,2}\l_v(P_j) -\veps \sum_{v \in T} d_v\max_{j=1,2}\l_v(P_j) \\
&\geq (1-2\veps) \sum_{v \in T} d_v\max_{j=1,2}\l_v(P_j).
\end{split}
\end{equation*}
Similarly for $\sigma_2$
\begin{equation*}
\begin{split}
&\sigma_2=\sum_{v\in S/T}^{good}d_v \big(\min_{j=1,2}\l_v(P_j)-(1-2\veps)\max_{j=1,2}\l_v(P_j) \big) > 0\\
\end{split}
\end{equation*}
with $\veps$ being small enough. Combining estimates for $\sigma_1, \sigma_2$ and using the fact that
$$\sum_{v \in T} d_v \max_{j=1,2} \l_v(P_j) \geq \max_{j=1,2} \sum_{v \in T} d_v \l_v(P_j)$$
one can see that
\begin{equation*}
\begin{split}
\sum_{v \in S}d_v \l_v(P_1-P_2) &\geq (1-2\veps) \sum_{v \in T} d_v\max_{j=1,2}\l_v(P_j) +(1-2\veps) \max_{j=1,2} \sum_{v \in S/T} d_v\l_v(P_j) \\
&\geq (1-2\veps) \max_{j=1,2} \sum_{v \in S} d_v \l_v(P_j).
\end{split}
\end{equation*}
Since $S$ contains all places that do divide the discriminant, then we have
$$\l_v(P) = \frac{1}{2} \log^+(|x(P)|_v)=0, \text{ for } v \notin S.$$ 
Then
$$\h_K(P_1-P_2) \geq (1-2\veps)\max_{j=1,2} \h_K(P_j) + \sum_{\bfrac{v \text{ finite}}{v(I) > 0}}d_v \l_v(P_1-P_2).$$
It remains to consider only finite places $v$, such that $v(I) > 0$. Let $\mathfrak{p}_v$ be the corresponding prime ideal in $O_K$ with its multiplicity $n_v$ in $I$. By reduction modulo $\mathfrak{p}_v^{n_v}$ our point $P_1-P_2$ becomes an origin $O$. Then 
$$v(x(P_1-P_2)) \leq -2n_v$$ 
and 
$$\l_v(P_1-P_2) \geq \frac{n_v}{e_v} \log p_v,$$
where $e_v$ is the ramification degree of $K_v$ and $p_v$ is the rational irreducible element under $v$.
Thus
$$\sum_{\bfrac{v \text{ finite}}{v(I) > 0}}d_v \l_v(P_1-P_2) = \log NI.$$
\end{proof}
We are going to exploit Lemma \ref{proposition34} to give an upper bound on the number of $S$-integral points. In order to get a good constant $C$, that appears in the main result of this paper we are going to apply sphere packings. We first subdivide the set of integer points on $E$ into "good slices" and then apply sphere packing bounds to each part separately.
Here we use the remarkable result of Kabatiansky and Levenstein (see, for example, \cite{Kabatjanskii1978}).

\begin{lemma}\label{kab-lev} [Kabatiansky-Levenstein \cite{Kabatjanskii1978}]
Let $A(n, \theta)$ be the maximal number of points that can be arranged on the unit sphere of $\R^n$ such that the angle between $P_1$, $O$ and $P_2$ for any two $P_1, P_2$ of them is no smaller than $\theta$. Then for $0 < \theta < \frac{\pi}{2}$
$$\frac{1}{n} \log_2 A(n, \theta) \leq \frac{1+\sin \theta}{2\sin \theta} \log_2 \frac{1+\sin \theta}{2\sin \theta} - \frac{1-\sin \theta}{2\sin \theta} \log_2 \frac{1-\sin \theta}{2\sin \theta} + o(1),$$
where the convergence is uniform and explicit for $\theta$ within any closed subinterval of $\left(0, \frac{\pi}{2}\right)$. In particular, for $\theta = \frac{\pi}{3}$, we have
$$\frac{1}{n}\log_2A(m,\theta) \leq 0.40141\ldots$$
\end{lemma} 

\begin{lemma} \label{slicepack35}
Let $c_1,c_2$ be two positive real numbers, $0 < \veps < \frac{1}{2}$, $n$ is a non-negative integer. For $\vec{X}=(X_i)_{1 \leq i \leq n} \in \F_q^n[T]$ consider 
$$S = \{\vec{X} \in \F_q^n[T] \; c_1 \leq |\vec{X}| \leq c_2\},$$ where $|\vec{X}|=\sum_{i=1}^{n} |X_i| = \sum_{i=1}^{n} q^{\deg X_i}$. Then there is a subset $T \subset \F_q^n[T]$ such that
$$\#T \leq C^n \veps^{-(n+1)}\left(1+\log \frac{c_2}{c_1}\right),$$
where the implied absolute constant $C$ is explicit and the balls $B(\vec{Y}, \veps |\vec{Y}|)$ cover all of $S$ for $\vec{Y} \in T$.
\end{lemma}
\begin{proof}
It is enough to show the covering by balls $B(\vec{Y}, 2\veps |\vec{Y}|)$. We wish to slice $S$ into a union of regions where $|\cdot|$ is almost constant, namely 
$$T = \bigcup_{0 \leq m \leq M} \frac{c_1 \veps (1+\veps)^m}{n} T_m,$$
where
$$T_m = \{\vec{Y} \in \F_q^n[T]: \; \frac{n}{\veps}(1-\veps) \leq |\vec{Y}| \leq \frac{n}{\veps}(1+\veps)\}
\;\text{ and }\; 
M = \log_{1+\veps}\log \frac{c_2}{c_1}.$$
Let $\vec{X} \in S$. Consider
$$m(\vec{X}) = \left\lfloor \log _{1+\veps} \frac{|\vec{X}|}{c_1} \right\rfloor 
\text{ and }
\vec{Z}(\vec{X}) = \left \lfloor \frac{n \vec{X}}{c_1 \veps (1+\veps)^{m(\vec{X})}}\right \rfloor,$$
where $\lfloor \cdot \rfloor$ is the floor function.
Define $\vec{Y} = \frac{c_1 \veps (1+\veps)^m}{n} \vec{Z}(\vec{X})$. Then
\begin{equation*}
\begin{split}
&|\vec{Z}(\vec{X})| \leq \frac{n |\vec{X}|}{c_1 \veps (1+\veps)^{m(\vec{X})}} \text{ and thus } |\vec{Y}| \leq |\vec{X}| < c_2,\\
&|\vec{Z}(\vec{X})| \geq \frac{n |\vec{X}|}{c_1 \veps (1+\veps)^{m(\vec{X})}}-1 \text{ and thus } |\vec{Y}| \geq 
|\vec{X}| - \frac{c_1 \veps (1+\veps)^{m(\vec{X})}}{n} \\
&\geq c_1 - \frac{c_1 \veps (1+\veps)^{M)}}{n} \geq c_1 - \frac{c_2 \veps}{n} > c_1.
\end{split}
\end{equation*}
We have just shown that given an $\vec{X} \in S$ one can find a point $\vec{Y}$, that depends on $\vec{X}$ and lies in $T$.
In addition $\vec{Y}$ has the following property 
$$d(\vec{X}, \vec{Y}) = |\vec{X} - \vec{Y}| \leq 2 \veps |\vec{Y}|,$$
where $d(\cdot, \cdot)$ is the associated metric. It remains to estimate the size of $T$
\begin{equation*}
\begin{split}
\#T &  \leq \left( 1+ \log_{1+\veps}\frac{c_2}{c_1}\right) \# T_m \leq  \left( 1+ \log_{1+\veps}\frac{c_2}{c_1}\right) \frac{(n(1+\frac{1}{\veps})+n)^n}{n!}.
\end{split}
\end{equation*}
The result follows after application of Stirling formula.
\end{proof}

We will need the following lower bound for a canonical height on $E$.
\begin{lemma} \label{lowerbnbcanheight}
Let $E$ be an elliptic curve over $K$. There is an absolute constant $0<c<1$ such that, for every non-torsion point $P\in E(K)$ we have the bound
$$\h(P)  > c^{m} \max \big(1, h(j(E))\big),$$
where $m$ is the number of mulpiplicative places and $j(E)$ is as usual a $j$-invariant of $E$. 
\end{lemma}
\begin{proof}
This Lemma is an analogous result to the ones in \cite{Silverman} and \cite{Hindry1988}. In fact, a stronger result was proven in \cite{Hindry1988}, namely:  $\hat{h}(P) \geq c \sigma_E h(E)$, where $\sigma_E$ is the Szpiro ratio (it gives $\hat{h}(P) \geq c_1 h(E)$ when $j(E) \in \F_q(T) /\F_q(T^p)$).
\end{proof}

\section*{Bounding the number of $S$-integral points}
In this section we prove the bound for the number of $S$-integer points on $E/K$ of height less than $h_0$. Here $t$ is a parameter to be optimized further. Then we are going to present a proof of the main result. It follows the way proposed in \cite{Helfgott2004}, \cite{Silverman}, \cite{GrossSilv} and later improved in \cite{Helfgott2006}. By embedding  $E(K)/E(K)_{tors}$ into $E(K) \otimes_{\Z} \R \cong \R^{\rank E}$ we can take the canonical height on $E$ to be squared Euclidean norm. The key idea consists of the fact that the points we are looking at  have large distance between each other. Namely, by choosing a good division of the area into small symmetric slices we can say that any two points are separated by almost 60 degrees. Then the number of integral points on $E$ is bounded above by $2^{\rank E}$ (this constant was later improved to $(1+\veps)$ in \cite{Helfgott2006}). It remains to apply Theorem \ref{bsdff} and (\ref{anbound}) for getting the result. 

\begin{theorem} \label{betaparam}
Let $E$ be an elliptic curve over $K$. Let also $S$ be a finite set of places of $K$, including all irreducible divisors of the discriminant of $E$. Then, for any $h_0 \geq 1$ and every $0\leq t \leq 1$, the number of $S$-integer points $P$ of $E(K)$ with a canonical height $\h(P) \leq h_0$ is at most
$$O \left( C^{|S|} \veps^{-2(|S|+[K:L])}|S|^{[K:L]} (1+\log h_0)^2 e^{t[K:L]h_0+(\beta(t)+\veps)\rank E}\right),$$ 
where $C$ is an absolute constant and $\beta(t)$ is defined for $0\leq t <1$ by
\begin{equation*}
\begin{split}
&\beta(t) = \frac{1+f(t)}{2f(t)}\log  \frac{1+f(t)}{2f(t)} -  \frac{1-f(t)}{2f(t)} \log  \frac{1-f(t)}{2f(t)},\\
&f(t)=\frac{\sqrt{(1+t)(3-t)}}{2}, \;\; \beta(1)=0.\\
\end{split}
\end{equation*}
\end{theorem}
\begin{proof}
Briefly speaking, we subdivide $S$-integer points on $E$ denoted by $E(K,S)$ into points $(\mmod I)$ for $I$ being a suitable ideal in $O_K$. Then Lemma \ref{proposition34} states that after some manipulations on this partition the points, that lie in the same class tend to be far away from each other in the Mordell-Weil lattice. Here we apply sphere packing bounds of Kabatiansky and Levenstein, namely Lemma \ref{kab-lev} to each part separately. These sphere packing bounds will bring us to the term  $e^{\beta(t)\rank E}$ on each part. Summation over all the classes gives rise to another term $e^{[K:L]h_0}$. We have only to take care of getting the right conditions to apply Lemma \ref{proposition34}.

We firstly subdivide $E(K,S)$ into a very few slices to force any two points of the same slice have comparable canonical height. Consider a set
$$\{P \in E(K,S): \h(P) \leq h_0\}.$$
We want to cover it by sets of the form
$$\{P \in E(K,S): (1-\veps)h_i \leq \h(P) \leq h_i\}.$$
By Lemma \ref{lowerbnbcanheight} it is enough to take $\ll \veps^{-1}(\log h_0 + |S|)$ such sets.
Then we are allowed to decrease the power of $(1+\log h_0)^2$ just to $1$, only for the set of points 
$$\{(1-\epsilon) h_0 \leq \hat{h}(P) \leq h_0\}.$$ 
Suppose first that $t \neq 0$. 
Let $S'$ be the set of places below $S$.
If 
$$X = \max (\lceil e^{th_0} \rceil, |\bar{S}|^{1+\frac{1}{[K:L]}}),$$
then there is an irreducible polynomial $f$ in $L$, such that $f \notin \bar{S}$ and $X \leq |f| \leq 2X$.
The ideal $I$ of $\Oo_{K}$ generated by $f$ satisfies
\begin{equation*}
\frac{\log N(I)}{[K:L]} \geq h_0 t, \;\;\; N(I) \ll_{[K:L]} s^{[K:L]+1}e^{t h_0 [K:L]}.
\end{equation*}
The $S$-integer points of our curve $E(K)$ fall into no more than $O_{[K:L]}(N(I))$ classes under the reduction modulo the corresponding ideal $I$. 
Define $R$ to be the set of all places of potentially multiplicative reduction. For any place $v \in R$ we subdivide the corresponding $E(K_v)$ into $n_v+1$ subsets, where $n_v$ is defined as in Lemma \ref{slicinglambda}(we take $\frac{\veps}{2}$ instead of $\veps$).
Consider arbitrary tuples of the form $(a_v)_{v \in R}$, $(b_v)_{v \in R}$, such that $0 \leq a_v \leq n_v$ and $b_v = 0,1$. We define $B$ as the set of non-torsion points $P \in E(K,S)$, such that for each $v \in R$ we have that $P$ falls into the corresponding $W$-class -- $P \in W_{v,a_v}$ and that $\l_v(p) \geq 0$ is equivalent to $b_v=1$. Now we bound the number of elements in 
$$B_{h_0} = \{P \in B: (1-\veps)h_0 \leq \hat{h}(P) \leq h_0\}.$$ 
The number of such sets $B$ is bounded above by $c_0^2 |\log \veps|^{s+[K:L]_{\veps}-2[K:L]}$, that brings us to the desired result. 
Define $M=(S-R) \cup \{v \in R: b_v =1\}$ and a map $l(P)=(d_v \l_v (P))_{v \in M}$. For $v \in S-M$ we know that $\l_v(P) < 0$, so that one can apply Lemma \ref{lowerbnbcanheight} and get 
$$|l(P)|_1 > [K:L] \kappa^s \max(1, h(j)).$$ 
Using \cite[Proposition 3]{GrossSilv} we get the bound 
$$\sum_{v \notin M} d_v \l_v(P) \geq -\frac{1}{24} h_k(j)-3[K:L].$$
On combining that we obtain
$$|l(P)|_1 \leq [K:L] (h_0+3+h(j)/24)$$ 
for $P \in B_{h_0}$. 
By Lemma \ref{slicinglambda} we can cover $l( B_{h_0})$ by at most
$$O(c_1^s \veps^{-(s+1)} \log (h_0+1))$$
balls  $B(x,\frac{\veps}{8}|x|_1)$ in the $1$-norm. Take two points $P_1, P_2 \in B_{h_0}$ with $l(P_i) \in B(x,\frac{\veps}{8}|x|_1)$ for $i=1,2$. We then have
$$|l(P_1)-l(P_2)|_1 \leq \frac{\veps}{4}|x|_1 \leq \frac{\veps}{2} \max_{j=1,2}|l(P_j)|_1.$$
If these points have the same reduction modulo $I$, then we apply Lemma \ref{proposition34} and get that
$$\hat{h}(P_1-P_2) \geq (1-\veps) \max_{j=1,2} \hat{h}(P_j)+ \frac{\log N(I)}{[K:L]} \geq (1+t-\veps)\max_{j=1,2} \hat{h}(P_j).$$
Now we embed the Mordell-Weil lattice modulo torsion into $\R^{\rank E}$ by taking $\hat{h}$ to be the square of the Euclidean height. Since all $\hat{h}(P_1)$, $\hat{h}(P_2)$, $\hat{h}(P_1-P_2)$ are positive, then the images of $P_1,P_2$, say, $Q_1,Q_2 \in \R^{\rank E}$ aree different from each other and from the origin, so that the angle between them is at least $\arccos \frac{1-t+O(\veps)}{2}$. 
We now apply Lemma \ref{kab-lev} and get that there are at most $e^{r(\beta(t)+O(\veps))} O_{[K:L]}(1)$ points of $B_{h_0}$ with an image in a given ball and with a prescribed reduction modulo $I$. Now we combine these results with the number of variants for $I$, the number of possible sets $B$ and the number of balls to get the theorem.
Notice, that in the case $t=0$ one simply proceeds without $I$. 
\end{proof}
The case $t=0$ is the pure application of sphere-packing results of Lemma \ref{kab-lev}, while the case $t=1$ is related to the corresponding result of Bombieri-Pila type.
\begin{corollary}
Let $E$ be an elliptic curve over $K$ defined by a Weierstrass equation with integer coefficients. Let $S$ be a finite set of places of $K$, including all places dividing the discriminant of $E$. Then for every sufficiently small $\veps$ the number of $S$ integral points on $E/K$ is at most
$$O_{\veps} \left( C^{s} \veps^{-2(s+1)} \left(\log |\Delta| + \log p\right)^2 e^{\rank E (\beta(0)+\veps)}\right).$$
\end{corollary}

We need as well upper bound for the canonical height. Here we adapt the result of Pacheco \cite{pacheco98}. There are known bounds over $\Q$, see, for example \cite{Hajdu1998}. Also one finds good bounds in \cite{Hindry1988}, but they work only in characteristic $0$.
\begin{lemma}
Let $E$ be an elliptic curve over $K$ defined by a Weierstrass equation $y^2=f(x)$. Let $\Oo_S$ be the ring of $S$-inetegers and $\Oo_S^*$ be the ring of $S$-units. Suppose that $f(X) \in \Oo_S$ and the discriminant $\Delta \in \Oo_S^*$, $p>2$. Define a set $\Xi$ in the following way.
Let $f(X) = (X-x_1)(X-x_2)(X-x_3)$ be the factorization of $f(X)$ in $\bar{K}[X]$. Let $P=(x_P,y_P) \in \Oo_S$. Define $\xi_i^2=X-x_i$, $i=1,2,3$. 
Let $L=K(x_1,x_2,x_3,\xi_1,\xi_2,\xi_3)$.
For any permutation $\{i,l,m\}$ of $\{1,2,3\}$ define
$$\Xi = \left\{\frac{(\xi_i-\xi_l)}{(\xi_i-\xi_m)}, \frac{(\xi_i-\xi_l)}{(\xi_i+\xi_m)}, 
\frac{(\xi_i+\xi_l)}{(\xi_i-\xi_m)}, \frac{(\xi_i+\xi_l)}{(\xi_i+\xi_m)}\right\}.$$
Then for any $\eta \in \Xi$ we have
$$\h_L(\eta) \leq 2p^{e}(2g_L-2+|S_L|),$$
where $S_L$ is the set of places of $L$ lying over $S$ and $g_L$ is the genus of $L$.
Moreover, if $p>3$, then for any $P=(x_P,y_P) \in \Oo_S$ we have
$\h_K(y_P^4/\Delta) \leq 48 p^{e} (2g-2+|S|)$.
\end{lemma}

We are now ready to give a version of Theorem \ref{betaparam} with an optimized parameter $t$.

\begin{corollary} \label{betaparamcor}
Let $E$ be an elliptic curve over a field $K$. Let $S$ be a finite set of places of $K$, that contains all places dividing the discriminant of $E$. Let $\alpha(x) = \min(xt+\beta(t), 0 \leq t \leq 1)$, where $\beta$ is as in Theorem \ref{betaparam}. Let also $R = \max (1, \rank E(K))$. Then for every $h_0 \geq 1$ and for every sufficiently small $\veps$, the number of $S$-integral points on $E$ over $K$, that have canonical height less or equal to $h_0$ is at most
$$O_{\veps, [K:L]} \left( C ^{\# S} \veps^{-2(\#S+[K:L])} \#S^{[K:L]}(1+\log h_0)^2 e^{R \alpha \left([K:L]h_0/R\right)+\veps R}\right),$$
where $C$ is an absolute constant.
\end{corollary}

We derive some quantitative bounds on the height of integral points on elliptic curve. We follow exactly the way proposed in \cite{Helfgott2006}. A combintation of a bound of Hajdu-Herendi \cite{Hajdu1998} together with our previous results gives the following.

\begin{corollary}
Let $E$ be an elliptic curve over a field $K$. Let $S$ be a finite set of places of $K$, that contains all places dividing the discriminant of $E$. Then the number of $S$-integral points on $E$ is at most
$$O_{\veps} \left( C ^{\# S} \veps^{-2(\#S+1)} (\log |f|+\log |\Delta |)^2 e^{(\beta(0)+\veps)} \rank E\right),$$
where $C$ is a constant, $f$ is the largest in norm element of $S$, $\Delta$ is the disciminant of $E$.
The calculation gives $\beta(0)=0.2782...$
\end{corollary}

Furthermore, in the same manner as in \cite{Helfgott2006} we obtain the next corollary.
\begin{corollary}
Let $\veps>0$, $E$ be an elliptic curve over a field $K$. Then the number of integral points on $E$ is at most
$$O_{\veps} \left(| \Delta |^{c+\veps}\right),$$
where $\Delta$ is the disciminant of $E$ and the constant $c=\frac{\beta(0)}{\log 2}=0.20070...$
\end{corollary}

\section*{Bounds on an algebraic rank}
Here we get the desired bound for an algebraic rank and give a bound for the number of $S$ integral points on $E$ in terms of its conductor. Due to the results of the previous section we have
\begin{equation*}
\begin{split}
\#E(K) &\ll c^{\rank E+m} \leq c^{\rank_{an} E} \\
&\leq \exp \left(\log c \left( \frac{(\deg N-8)\log q}{2 \log \deg N} + O \left(\frac{\deg N \log^2 q}{\sqrt{q} \log^2 \deg N}\right)\right)\right),
\end{split}
\end{equation*} 
where we used the fact that $\rank E \leq \rank_{an} E$ as well as the explicit formula given in Theorem \ref{brumer}.
We see that the term in $O(\cdot)$ is smaller than the main term, so we can simply rewrite
$$\#E(K) \ll c^{\rank E+m} \leq \exp \left(c \frac{\deg N \log q}{ \log \deg N}\right).$$ 

\subsection*{Comparison to Bombieri-Pila type bound}
Let $S$ be the set of all points of bad reduction of an elliptic curve $E/K$. Consider $h_0 > c \max (\deg \Delta, h(j))$, where $\Delta$ is the discriminant and $j$ is the $j$-invariant of $E/K$ for some constant $c$. 
The main contribution to Theorem \ref{betaparam} and, respectively, Corollary \ref{betaparamcor} is given by $e^{R\alpha(h_0/R)}$. The minimum in $\alpha$ is attained to the left of $t=1$.
Since $h_0 > c \deg \Delta$, then $\alpha(h_0/R) < (1-\delta_0)h_0/R$, where $\delta_0$ positive and depending only on $c$. Thus for any $\delta_1 \leq \delta_0$ we obtain a bound
$$\#E(K,S) \ll e^{(1-\delta_1)h_0},$$
while Bombieri-Pila type result brings us to $e^{h_0}$, thus this method gives an improvement in the exponent and also improves the corresponding results from \cite{helf-ven}.

Another possible way to get this sort of bounds is using the work of Bhargava et al. on bounding the size of 2-torsion group, see \cite{Bhargava2017}.
The authors of \cite{Bhargava2017} proved the first nontrivial bounds on the sizes of $2$-torsion subgroups of the class groups of cubic and higher degree number fields. This is also an improvement on the bounds on the number of integral points given in \cite{helf-ven}. They also gave a result for the function fields, see \cite[Theorem 7.1]{Bhargava2017}.  

\subsection*{Acknowledgements}
The author thanks Marc Hindry for useful remarks on the previous version of this paper. 
The author would also like to thank G\"{o}ttingen University for its hospitality while completing presented work.

\bibliography{../library}{}
\bibliographystyle{abbrv}
\end{document}